\documentclass[12pt, leqno]{article}
\usepackage{indentfirst}
\usepackage{amstext}
\usepackage{amsthm}
\usepackage{amsopn}
\usepackage{amsfonts}
\usepackage{amsmath}
\usepackage{latexsym}
\usepackage{amscd}
\usepackage{amssymb}
\usepackage{amsmath}
\usepackage{color}

\usepackage[all,cmtip]{xy}
\usepackage{leftidx}
\usepackage{graphicx}
\usepackage{tikz}
\usepackage{dutchcal}

=\oddsidemargin \font\teneufm=eufm10 \font\seveneufm=eufm7
\font\fiveeufm=eufm5
\newfam\eufmfam
\textfont\eufmfam=\teneufm \scriptfont\eufmfam=\seveneufm
\scriptscriptfont\eufmfam=\fiveeufm

\let\goth\mathfrak
\def\cR{\mathcal R}
\def\fp{\mathfrak{p}}
\def\fq{\mathfrak{q}}

\def\gg{\goth g}

\def\gg{\goth g}

\def\gp{\goth p}

\def\gg{\goth g}

\def\gq{\goth q}

\def\1{\mbox{\bf 1}}

 \DeclareMathOperator{\Hom}{Hom}
\DeclareMathOperator{\Aut}{Aut} 
 \DeclareMathOperator{\Der}{Der}

\DeclareMathOperator{\End}{End} 

\DeclareMathOperator{\Ad}{Ad} 
  \DeclareMathOperator{\GL}{\bf GL}

\newtheorem{theorem}{Theorem}[section]

\newtheorem{corollary}[theorem]{Corollary}

\newtheorem{lemma}[theorem]{Lemma}

\newtheorem{proposition}[theorem]{Proposition}

\theoremstyle{definition}
\newtheorem{remark}[theorem]{Remark}
\newtheorem{example}[theorem]{Example}

\numberwithin{equation}{section}

\def\Z{\Bbb Z}
\def\C{\Bbb C}
\def\Q{\Bbb Q}

\def\C{\mathbb C}

\def\bF{\text{\rm \bf F}}
\def\bH{\text{\rm \bf H}}
\def\bG{\text{\rm \bf G}}

\def\bW{\text{\rm \bf W}}

\def\cEnd{\mathcal End }

\def\bX{\text{\rm \bf X}}
\def\rX{\text{\rm X}}

\def\rU{\text{\rm U}}
\def\rV{\text{\rm V}}
\def\rZ{\text{\rm Z}}
\def\cL{{\cal L}}

\def\rL{{\rm L}}
\def\cO{{\cal O}}
\def\calO{{\cal O}}
\def\cF{{\cal F}}
\def\cE{{\cal E}}

\def\rLie{\text{\rm Lie}}

\def\ad{\text{\rm ad}}

\def\bT{\text{\rm \bf T}}

\def\bL{\text{\rm \bf L}}

\def\rS{\text{\rm{S}}}
\def\T{\text{\rm{T}}}

\def\bO{\text{\rm \bf O}}

\def\bV{\text{\rm \bf V}}
\def\bG{\text{\rm \bf G}}
\def\bH{\text{\rm \bf H}}

\def\us{\underset}

\def\ol{\overline}

\def\rLie{\text{\rm Lie}}

\def\Lie{\mathop{\rm Lie}\nolimits}

\def\2int{\mathop{2\int}\nolimits}

\def\End{\mathop{\rm  End}\nolimits}

\def\dim{\mathop{\rm dim}\nolimits}

\def\Spec{\mathop{\rm Spec}\nolimits}

\def\Lie{\mathop{\rm Lie}\nolimits}
\def\bLie{\mathop{\bf Lie}\nolimits}

\def\Hom{\mathop{\rm Hom}\nolimits}
\def\Der{\mathop{\rm Der}\nolimits}

\def\Aut{\text{\rm{Aut}}}

\def\bAut{\text{\bf{Aut}}}
\def\bIsom{\text{\bf{Isom}}}
\def\bOut{\text{\bf{Out}}}
\def\bEnd{\text{\bf{End}}}
\def\bDer{\text{\bf{Der}}}

\def\cLie{\mathcal{L}\mathcal{i}\mathcal{e}}

\def\cAut{\mathcal{A}\mathcal{u}\mathcal{t}}

\def\resp.{\mathop{\rm resp.}\nolimits}

\def\lgr{\longrightarrow}
\def\la{\longleftarrow}

\font\math=cmmi10
\def\varpi{\hbox{\math\char'44}}
\def\simlgr{\buildrel\sim\over\lgr}
\def\simla{\buildrel\sim\over\la}

\def\pa{\S\kern.15em }

\def\un{\uppercase\expandafter{\romannumeral 1}}
\def\deux{\uppercase\expandafter{\romannumeral 2}}
\def\trois{\uppercase\expandafter{\romannumeral 3}}
\def\quatre{\uppercase\expandafter{\romannumeral 4}}
\def\cinq{\uppercase\expandafter{\romannumeral 5}}
\def\six{\uppercase\expandafter{\romannumeral 6}}

\def\gg{\goth g}

\newcommand{\uV}{\mathbf V}

\def\hfl#1#2#3{\smash{\mathop{\hbox to#3{\rightarrowfill}}\limits
^{\scriptstyle#1}_{\scriptstyle#2}}}
\def\gfl#1#2#3{\smash{\mathop{\hbox to#3{\leftarrowfill}}\limits
^{\scriptstyle#1}_{\scriptstyle#2}}}

\begin{document}

\title{Fiberwise Criteria for Twisted Forms of Algebraic Structures}

\author{ P. Gille$^{\rm 1}$ and A. Pianzola$^{\rm 2,3}$}

\date{}
\maketitle

\centerline{\it In memoriam Professor Georgia Benkart}

\bigskip

$^{\rm 1}${\it UMR 5208 du CNRS -
Institut Camille Jordan - Universit\'e Claude Bernard Lyon 1, 43 boulevard du
11 novembre 1918, 69622 Villeurbanne cedex - France.}

$^{\rm 2}${\it Department of Mathematical Sciences, University of
Alberta, Edmonton, Alberta T6G 2G1, Canada.}

$^{\rm 3}${\it Centro de Altos Estudios en Ciencia Exactas, Avenida de Mayo 866, (1084) Buenos Aires, Argentina.}

 \maketitle 
 
 \begin{abstract} We provide a criterion for certain algebraic objects over Jacobson schemes to be forms of each other based on their behaviour at closed fibres. This criterion permits to answer a question that I. Burban had asked the authors.
 

\end{abstract}

\bigskip
\bigskip

\section{Introduction}
Let $\bG$ be a group scheme over a base scheme $\rS.$ The concept of $\bG$--torsors over $\rS$ is to be found in many areas of mathematics and mathematical physics; they are useful tools to frame problems in a language that is partial to powerful methods from algebraic geometry. $\bG$ acts on itself by right multiplication. A $\bG$--torsor is a scheme $\bX$ with a right action of $\bG$ that ``locally" looks like $\bG$ with this action. By locally we mean that there exists a faithfully flat and locally finitely presented  scheme morphism $\T \to \rS$ such that $\bX_\T $ and $\bG_\T$ are isomorphic as $\T$--schemes with the corresponding induced $\bG_\T$--action.

The treatise \cite{SGA3} shows how a deep understanding of $\bG$ can be had through the study of its fibres $\bG_s := \bG \times_\rS \Spec\big(\kappa(s)\big)$ where $\kappa(s)$ is the residue field of  $s \in \rS.$ It thus seems natural, given a right action of $\bG$ on an $\rS$--scheme $\bX,$ to study situations under which  the fibres $\bX_s$ being $\bG_s$--torsors yield information about $\bX$ itself being a $\bG$--torsor. In practice the fibres of closed points are more tractable. The existence of ``enough" closed points brings up the natural working hypothesis that $\rS$ be a Jacobson scheme. It is under this  assumption that we are able to create a fiberwise criterion for torsors under certain reductive group schemes.

This work grew up out of the desire to answer a question that I. Burban posed to the authors a couple of years ago \cite{Bu}. His question and the (positive) answer that can be given using our main result is given in the last section.

\medskip

\noindent {\bf Acknowledgments.} The second author wishes to sincerely thank CNRS for facilitating an invited researcher visit to the  Camille Jordan
Institute.

\section{Preliminaries on  group schemes and smoothness}

Throughout this paper $\rS$ is a (base) scheme with structure sheaf $\calO_\rS$. By an $\rS${\it ---functor} (resp.\ {\it  monoid, group}) we mean a contravariant functor $\bF$
from the category of schemes over $\rS$ to the category of sets (resp.\ monoids, groups). When $\bF$ is representable, i.e. a scheme, we say that $\bF$ is an $\rS$--{\it scheme} (resp. {\it monoid scheme, group scheme}). The use of bold face characters in the text is used to emphasize the functorial nature of the object under consideration. In particular, if ${\rm X}$ is an $\rS$--scheme then its functor of points
$$\T \mapsto \Hom_{\rS-{\rm Sch}}(\T ,{\rm X})$$
will be denoted by $\bX.$

If $\bF$ is an $\rS$--functor and $\T$ a scheme over $\rS$, we denote by $\bF_\T$ the $\T$--functor obtained by base change. If $\T = \Spec(R)$ we denote $\bF(\T)$ and $\bF_\T$ by $\bF(R)$ and $\bF_R$ respectively. If $\bF$ is an $\rS$--scheme, then $\bF_\T = \bF \times_\rS \T$ and, as it is customary, we denote  $\bF \times_{\Spec(R)} \T$ by  $\bF \times_R \T$ and $\bF \times_\rS \Spec(R)$ by  $\bF \times_\rS R$

We denote as usual by $ \bO_\rS $ the affine $\rS$--ring scheme $\rS[t] =  \Spec(\Z[t])\times_\Z \rS.$  Thus $\bO_\rS(\T) = \cO_\T(\T).$ Recall that an $\bO_\rS$-{\it module} is an abelian $\rS$--group ${\bf M}$ together with  an $\bO_\rS(\T)$--module structure on ${\bf M}(\T)$ that is functorial on $\T.$ The concept of an $\bO_\rS$--algebra is defined similarly. Base change is defined and denoted as it is for functors.

We mainly use  the terminology and notation of Grothendieck--Dieudonn\'e \cite{EGA-neu}, which for the most part agrees with that of Demazure--Grothendieck used in \cite[Exp.~I and II]{SGA3}. Below we briefly review those concepts and results that are relevant to this paper. \smallskip

\subsection{Groups attached to quasi--coherent modules} Let $\cE$ be a quasi--coherent module over $\rS$. We denote its 
dual by $\cE^{^{\vee}}$. For each morphism  $f: \T \to \rS$ we let $\cE_{\T}=f^*(\cE)$ be the 
inverse image of $\cE$ under the morphism $f$, and define an abelian $\rS$--group
$\bV(\cE)$ by $\bV(\cE)(\T) = \Hom_{\calO_\T}(\cE_\T, \calO_\T) = \Gamma\big(\T, (\cE_\T)^{^{\vee}}\big)$.  $\bV$ is actually an $\rS$--scheme; it is represented by the affine scheme $\Spec\bigl( \mathrm{\bf Sym}(\cE)\bigr)$ where $\mathrm{\bf Sym}(\cE)$ is the symmetric $\cO_\rS$--algebra of $\cE$ \cite[9.4.9]{EGA-neu}.
If $\cE$ is of finite type (resp.\ of finite presentation), then $\uV(\cE)$ is an $\rS$--scheme of
finite type (resp.\ of finite presentation), ibid, 9.4.11.

The  abelian $\rS$--group $\bW(\cE)$ is defined by $ \bW(\cE)(\T) = \Gamma(\T , \cE_{\T}).$ Recall that if $\cE$ is locally free of finite type then $\bW(\cE) \simeq \bV(\cE^{^{\vee}}).$ In particular $\bW(\cE)$ is in this case an affine $\rS$--scheme. Note that the abelian $\rS$--groups $\bV(\cE)$ and $\bW(\cE)$ have natural $\bO_\rS$--module structures.

\begin{example}\label{affineO} Assume $\rS = \Spec(R).$ Let $\cL$ be a quasi--coherent $\cO_\rS$--module and denote by $\rL$ the corresponding $R$--module. We denote $\bW(\cL)$ by $\bW(\rL).$ For all $\rS$--scheme $\T$ by definition $\bW(\rL)(\T)$ is the $\cO_\T(\T)$--module $\rL\otimes_R \cO_\T(\T).$ Similarly if $\cL$ is an $\cO_\rS$--algebra.
\end{example}

\begin{remark}\label{FuFa} We can view $\bV$ (resp. $\bW$) as a contravariant (resp. covariant) functor from the category of quasi-coherent $\cO_\rS$-modules to the category of $\bO_\rS$-modules.
These functors are full and faithful  \cite[I Prop. 4.6.2]{SGA3}
\end{remark} 

To $\cE$ we attach the $\rS$-functor $\bEnd_{\cO_\rS-{\rm mod}}(\cE)$ whose functor of points is given by
$ \T \mapsto \End_{{\cO_\T}-{\rm mod}}(\cE \otimes_{\cO_\rS} \cO_\T).$  This can be viewed as an abelian $\rS$-group or an $\bO_\rS$--module. We also have the $\rS$--group $\bAut_{\cO_\rS-{\rm mod}}(\cE)$ given by $\T \mapsto \Aut_{{\cO_\T}-{\rm mod}}(\cE \otimes_{\cO_\rS} \cO_\T).$  If in addition $\cE$ is an $\cO_\rS$--algebra, one defines the $\rS$--group $\bAut_{\cO_\rS-{\rm alg}}(\cE)$ in the obvious way. It is an $\rS$--subgroup of $\bAut_{\cO_\rS-{\rm mod}}(\cE)$. 

Similarly to an $\bO_\rS$-module $\bL$ we attach an abelian $\rS$--group (in fact an $\bO_\rS$--module) $\bEnd_{\bO_\rS-{\rm mod}}(\bL)$ via $\T \mapsto \End_{\bO_\T-{\rm mod}}(\bL_\T).$ Finally if $\bL$ is an $\bO_\rS$--algebra, the $\rS$--group $\bAut_{\bO_\rS-{\rm alg}}(\bL)$ is given by the functor of points $\T \mapsto  \Aut_{\bO_\T-{\rm alg}}(\bL_\T).$

\begin{lemma}\label{repres} Let $\cL$ be an $\cO_\rS$--algebra. Assume that as an $\cO_\rS$--module $\cL$ is locally free of finite rank. Then.

\noindent (1) The natural maps
$$ \bEnd_{\cO_\rS-{\rm mod}}(\cL) \to \bEnd_{\bO_\rS-{\rm mod}}\big(\bW(\cL)\big)$$
and
$$ \bAut_{\cO_\rS-{\rm alg}}(\cL) \to \bAut_{\bO_\rS-{\rm alg}}\big(\bW(\cL)\big)$$
are $\rS$-functor isomorphisms.
\smallskip

\noindent (2)  $\bEnd_{\cO_\rS-{\rm mod}}(\cL)$ and $\bAut_{\cO_\rS-{\rm alg}}(\cL)$ are representable by affine $\rS$--schemes of finite presentation.
\end{lemma} 

\begin{proof} (1) From their definition we see that the maps under consideration are functorial . That they are bijective on points follows from  Remark \ref{FuFa}.
\smallskip

(2) Assume that $\rS = \Spec(R).$ Then $\cL$ corresponds to an $R$-module $\rL$ which is projective of finite rank.  The $\cO_\rS$--module $\cEnd_{\cO_\rS-{\rm mod}}(\cL)$ corresponds to the $R$-module $\End_{R-{\rm mod}}(\rL) \simeq \rL^* \otimes_R \rL .$ Because the $R$-module map $\End_{R-{\rm mod}}(\rL) \otimes_R R' \to
\End_{R'-{\rm mod}}(\rL\otimes_R R')$ is an isomorphism for all $R'/R,$ it follows that $\bEnd_{\cO_\rS-{\rm mod}}\big(\cL)$ is represented by the affine $R$-scheme of finite presentation $\bW(\rL^*\otimes_R \rL).$ 

It is clear that $\bAut_{\cO_\rS-{\rm alg}}(\cL)$ is a closed subscheme of $\bEnd_{\cO_\rS-{\rm mod}}(\cL),$ hence also affine, which is of finite presentation since $L$ is locally free of finite rank.

From the foregoing considerations it follows that our two functors are affine $\rS$-schemes which are locally of finite presentation. Since their structure morphisms are affine, they are quasi-compact and separated, hence of finite presentation.
\end{proof}

\begin{remark}\label{warning}  Let $\cL$ be a quasi-coherent $\cO_\rS$--module, and consider the corresponding $\cO_\rS$--module $\cEnd_{\cO_\rS-{\rm mod}}(\cL).$ There is a natural $\bO_\rS$--module morphism
$$\bW\big(\cEnd_{\cO_\rS-{\rm mod}}(\cL)\big) \to \bEnd_{\bO_\rS -{\rm mod}}\big(\bW(\cL)\big).$$
This morphism {\it need not} be an isomorphism. It is if $\cL$ is locally free of finite type.
\end{remark}

\subsection{Group schemes and Lie algebras} Throughout $\bG$ will denote an $\rS$-group scheme and we denote by $e \in \bG(\rS)$
its unit section. We refer to \cite[I and II]{SGA3}, \cite{DG}  and \cite[\S 1]{LLR} for details in what follows.

If $\T$ is a scheme we denote by $\T[\epsilon]$ the corresponding scheme of dual number \cite[II.2]{SGA3}. The functor $\T \to \bLie(\bG)(\T) := \ker\big( \bG(\T[\epsilon]) \to \bG(\T)\big)$ is an $\bO_\rS$--module isomorphic to $\bV(\omega^1_{\bG/\rS})$ where $\omega^1_{\bG/\rS}= e^*( \Omega^1_{\bG/\rS})$ ibid. Prop. 3.3 and 3.6. Furthermore, the $\bO_\rS(\T)$--module $\bLie(\bG)(\T)$ has a natural Lie algebra structure. It is thus an $\bO_\rS$--Lie algebra. Recall that for all scheme morphisms $\T \to \rS$ we have a natural $\bO_\T$--Lie algebra isomorphism
\begin{equation} \label{Lialg}
\bLie(\bG) \times_\rS \T \simeq \bLie(\bG_\T).
\end{equation} 
The $\cO_\rS(\rS)$--Lie algebra $\bLie(\bG)(\rS)$ is denoted by $\Lie(\bG).$ From the above isomorphism we have
\begin{equation} \label{Lialg1}
\bLie(\bG)( \T) = \Lie(\bG_\T).
\end{equation} 

We denote by $\cLie(\bG)$ the vector group sheaf (fibration vectorielle) of sections of the affine scheme $\bLie(\bG) \to \rS.$ In other words,  $\cLie(\bG)$ is the  $\cO_\rS$-module $({\omega^1_{\bG/\rS}})^{\vee} = \Hom_{\cO_\rS}(\omega^1_{\bG/S}, \cO_\rS).$ Note that $\cLie(\bG)$ is naturally an $\cO_\rS$--Lie algebra.

\begin{remark}\label{Liesmooth} 
In general $\cLie(\bG)$ does not determine $\bLie(\bG),$ but it does if $\omega^1_{\bG/\rS}$ is locally free of finite type, in particular if $\bG$ is smooth. In this case  $\bLie (\bG) = \bW\big(\cLie(\bG)\big).$ See \cite[II Lemma 4.11.7]{SGA3}.  If $\rS = \Spec(R)$ the $R$--Lie algebra $\rLie(\bG) = \cLie(\bG)(\rS)$ is a locally free $R$--module of finite type. For any ring extension $R'/R$ we have $\rLie(\bG_{R'}) = \bLie(\bG)(R') = \rLie(\bG) \otimes_R R'$.
\end{remark} 
\medskip



\begin{lemma}\label{derLie} Let $\cL$ be an $\cO_\rS$-algebra which is locally free  of finite type (as an $\cO_\rS $--module). Let $\bL = \bW(\cL).$ There is a natural $\bO_\rS$-Lie algebra isomorphism $\bDer_{\bO_\rS-{\rm alg}}(\bL) \simeq \bLie\big(\bAut_{\bO_\rS-{\rm alg}}(\bL)\big).$
\newline
  \end{lemma}
\begin{proof} Since $\bL$ is a good $\bO_\rS$-module \cite[II Def. 4.4 and Ex. 4.4.2]{SGA3}, we have an $\bO_\rS$-module isomorphisms $\eta: \bEnd_{\bO_\rS-{\rm mod}}(\bL) \simeq \bLie\big(\bAut_{\bO_\rS-{\rm mod}}(\bL)\big)$ \cite[II Prop. 4.5]{SGA3}. Since $\cL$ is locally free of finite type we can appeal to Lemma \ref{repres}(1) to conclude that this is in fact an isomorphism of $\rS$-schemes. We claim that the restriction of $\eta$ to $\bDer_{\bO_\rS-{\rm alg}}(\bL)$ is our isomorphism. The proof reduces to the case when $\rS$ is affine, which can be found in \cite[II \S4 2.3]{DG}.
\end{proof}

\subsection{ Connected component of the identity, forms and type of a reductive group scheme.}\label{CCFT}


Let $\bG$ be an $\rS$-group scheme which is locally of finite presentation.
We consider the $\rS$--subgroup (functor) $\bG^\circ$
of $\bG$ \cite[VI$_B$.3.1]{SGA3} defined by
$$ 
\bG^\circ(\T) = \Bigl\{ u \in  \bG(\T) \, \mid \, \forall s \in \rS, u_s(\T_s) 
\subset \bG_s^\circ \Bigr\}.
$$
where $ \bG_s^\circ$ is the connected component of the identity of the $\kappa(s)$-algebraic group $\bG_s.$ If $\bG$ is  smooth along the unit section, $\bG^\circ$ 
is representable by a smooth $\rS$--group scheme  called {\it the connected component of the identity of $\bG$} \cite[VI$_{\rm B}$.4.1]{SGA3}. Furthermore,  the fibre $(\bG^{\circ})_s$ is naturally isomorphic to  $\bG_s^\circ.$

Let $\bG$ be a reductive $\rS$--group scheme, that is a smooth and affine $\rS$--group whose geometric fibres are {\it connected} reductive algebraic groups. An $\rS$--group scheme $\bG'$ is called a {\it (twisted) form of} $\bG,$ if there exist a faithfully flat and localy presented extension $\T \to \rS$ such that $\bG_\T$ and $\bG'_\T$ are isomorphic $\T$--group schemes. In a similar fashion one defines forms of $\bO_\rS$--algebras. 

Let $\bG$ be a reductive $K$-group scheme, where $K$ is an algebraically closed field. If $\bT$ is a maximal torus of $\bG$, then $\bT$ is split and defines a root datum  which is up to isomorphism independent of the choice of $\bT.$ It is called the {\it type of } $\bG$ and is denoted by $\cR(\bG).$

Let $\bG$ be a reductive $\rS$--group scheme. If $s \in \rS,$ the {\it type  of  $\bG$ at $s$} is the type of the reductive $\overline{\kappa(s)}$--group $\bG_{\overline{s}} := \bG_s \times_{\kappa(s)} \overline{\kappa(s)}.$ Because $\bG$ posseses maximal tori locally for the \'etale topology, one knows that the type function $s \mapsto \text{ type of}  \,\, \bG_{\overline{s}}$ is locally constant 
\cite[XXII 2.8]{SGA3}. We say that $\bG$ is of {\it constant type} if the type function is constant. This is the case, for example, if $\rS$ is connected.

\subsection{Specific properties in characteristic zero}\label{subsec_specific}\label{2.4}
We assume in this section that $\rS$ is a $\Q$--scheme.
Let $\bG$ be a semisimple adjoint $\rS$--group scheme of constant 
 type. We will denote the $\bO_\rS$--Lie algebra $\bLie(\bG)$ by $\bL,$ and the $\cO_\rS$--Lie algebra  $\cLie(\bG)$ by $\cL.$ Recall (see Remark \ref{Liesmooth})
that $\cL$
is locally free of finite  type and that $\bW(\cL) = \bL.$ 
Since $\bG$ and its simply connected cover $\bG^{\rm sc}$ have the same Lie algebra,  we have  isomorphisms
 of affine $\rS$--group schemes \cite[XXIV 3.6 and  7.3.1]{SGA3}
 \begin{equation}\label{AutG=AutL}
 \bAut(\bG) \simeq \bAut(\bG^{\rm sc})  \simeq \bAut_{{\bO_\rS}-\rLie}(\bL).\footnote{\, La ``définition évidente" of the $\rS$-group scheme $\underline{\cAut}_{\cO_\rS-alg-de-Lie}(\cL)$  of 7.3.1(iii) is precisely our $\bAut_{\cO_\rS-{\rm Lie}}(\cL),$ so that the asertion is that the natural map $\bAut(\bG^{\rm sc}) \to \bAut_{\cO_\rS-{\rm Lie}}(\cL)$ is an isomorphism. Finally  $\bAut_{\cO_\rS-{\rm Lie}}(\cL) \simeq \bAut_{\bO_\rS-{\rm Lie}}(\bL)$ by Lemma \ref{repres}.} 
\end{equation} In particular, the $\rS$-scheme $\bAut_{{\bO_\rS}-\rLie}(\bL)$ is smooth and affine. By  (\ref{AutG=AutL}) and  \cite[XXIV 1.3 and 1.8]{SGA3} we have an isomorphism
\begin{equation}\label{G}
\bG \simeq \bAut_{{\bO_\rS}-\rLie}(\bL)^{\circ}.
\end{equation} This yields an $\bO_\rS$-Lie algebra isomorphism
\begin{equation}\label{identif}
\bL \simeq \bLie\big(\bAut_{{\bO_\rS}-\rLie}(\bL)\big) = 
\bLie\big(\bAut_{{\bO_\rS}-\rLie}(\bL)^\circ\big).
 \end{equation}

Recall  \cite[II Theo. 4.7 and Prop. 4.8]{SGA3} the adjoint representation \break
 $\Ad : \bG \to \bAut_{{\bO_\rS} - {\rm mod}}(\bL)$ and the induced a $\bO_\rS$-Lie algebra homomorphism
$ \ad:  \bL \to \bEnd_{{\bO_\rS}- {\rm mod}}(\bL)$. 

\begin{lemma}\label{MZ} Under the identification $\bL \simeq \bLie\big( \bAut_{{\bO_\rS}-\rLie}( \bL)\big)$ of (\ref{identif}) the map $\ad$ induces an $\bO_\rS$-Lie algebra isomorphism between $\bL $ and $\bDer_{\bO_\rS - {\rm Lie}}(\bL)$.
\end{lemma}
\begin{proof} All the morphisms under consideration are $\rS$--scheme morphisms, so the question is local on $\rS$ and we may assume that $\rS = \Spec(R)$ for some $\Bbb{Q}$-ring $R.$ Consider the $R$-Lie algebra  $L =\bL(R).$ Since $\bL = \bW(\rL)$ where $\bW(\rL)(\T) = \rL \otimes_R \cO_\T(\T),$ it will suffice to show that for all ring extensions $R'/R$ the $R'$-Lie algebra homomorphism  $d =\ad_R' : \rL \otimes_R R'\to \text{\rm Der}_{R'-\rLie}(\rL \otimes_R R')$ is an isomorphism. By replacing $\bG$ by $\bG_R'$ we may assume that $R= R'.$ By \cite[II 4.7.2]{SGA3} the map $d$ is nothing but the adjoint representation $d(x) = \text{\rm ad}_\rL(x).$ Since $\bG$ is of constant type there exists a unique Chevalley group $\bG_0$ such that $\bG$ is a twisted form of $\bG_0 \times _\Bbb{Z} R.$ Let $\gg = \bLie(\bG_0)(\Bbb{Z}) \otimes _\Bbb{Z} \Bbb{Q}.$ Then $\gg$ is a finite dimensional split semisimple Lie algebra over $\Bbb{Q}$ and $\rL$  is a twisted form of $\gg \otimes_\Bbb{Q} R,$ that is, there exists an fppf extension (in fact \'etale cover) $R'$ of $R$ such that the $R'$-Lie algebras $\rL \otimes_R R' $ and $(\gg \otimes_\Bbb{Q} R)\otimes_R R' \simeq \gg \otimes_\Bbb{Q} R'$ are isomorphic. Fix one such isomorphism \begin{equation}\label{iso}
\psi : \rL \otimes_R R' \to \gg \otimes_\Bbb{Q} R'.
\end{equation}
 If $d(x) = 0,$ then $x$ belongs to the centre of $\rL$ and therefore $\psi(x \otimes 1)$ belongs to the centre of $\gg \otimes_\Bbb{Q} R',$ which is trivial. Since $R'/R$ is faithfully flat $x = 0$ so that $d$ is injective and we henceforth identify $\rL$ with an $R$-submodule of $\Der(\rL).$ To show that $d$ is surjective, that is that every derivation of $\rL$ is inner,  we reason as follows. If $\rL = \gg \otimes_\Q R$ every derivation is inner by Whitehead's Lemma (see  \cite[Example 4.9]{P1}).\footnote{\,  Note that our $\rL$ is denoted by  $\cL$ in \cite{P1}.} For $\rL$ arbitrary, we appeal to the isomorphism (\ref{iso}). By the split case we have $(\Der(\rL)/\rL)\otimes_R R' = 0.$ Thus $\rL = \Der(\rL)$ as desired. \end{proof}


\subsection{Semicontinuity property for Lie algebras }

For convenience we recall the following  well-known fact.

\begin{lemma} \label{lem_upper} Let $\cF$ be a quasi-coherent $\cO_\rS$--module of finite 
presentation.
Then the function  $s \mapsto \dim_{\kappa(s)}\big( \cF \otimes_{\cO_\rS} \kappa(s)\big)$
is upper semi-continuous.\end{lemma}

\begin{proof} The statement is local and therefore reduces then to the case of a ring $R$
 and of an $R$--module $M$ of finite presentation \cite[Tag 01PC]{Stacks}. 
If $R$ is noetherian then $M$  is coherent and  
the result can be found in  \cite[12.7.2]{Ha}.
For lack of a reference we provide a proof in the general case.

We consider an exact sequence $R^m \xrightarrow{A} R^n \to M \to 0$
where $A$ is a matrix of size $(m,n)$ with entries in $R.$ Let $\gq \in \Spec(R)$ and $r=\dim_{\kappa(\gq)} (M \otimes_R \kappa(\gq) ).$ Since the sequence
${\kappa(\gq)}^m \xrightarrow{A_{\kappa(\gq)}} {\kappa(\gq)}^n \to M_{\kappa(\gq)} \to 0$ is exact,
there  exists a minor $B$ of $A$ of size $(n-r,n-r)$ such that $\det(B_{\kappa(\fq)} )\neq 0 \in \kappa(\fq).$ Consider the basic open set $\rU$ of $\Spec(R)$ consisting of prime ideals of $R$ that do not contain $f = \det(B) \in R.$ From  $\det B_{\kappa(\fq)} \neq 0$ it follows that $\fq \in \rU.$ It is clear that for all $\gp \in \rU$ the image of $A_{\kappa(\gp)}$ is a $\kappa(\gp)$-space of dimension at least $n-r.$  Thus
$$\dim_{\kappa(\gp)} \big(M \otimes_R \kappa(\gp) \big)
\leq r = \dim_{\kappa(\gq)} \big(M \otimes_R \kappa(\gq) \big)$$
for all $\gp \in \rU,$  so that our function is thus upper semi-continuous as desired.
\end{proof}

\begin{lemma} \label{lem_upper_Lie} Let $\bG$ be an $\rS$--group scheme that is locally of finite presentation.
Then $\omega^1_{\bG/\rS}$ is an $\bO_\rS$--module of finite presentation and
the map $s \mapsto \dim_{\kappa(s)}\Lie(\bG_s)$
is upper semi-continuous.
\end{lemma}

\begin{proof} That $\omega^1_{\bG/\rS}$ is of finite presentation follows from  \cite[Tag 01V2,  01V3]{Stacks}.
 
By (\ref{Lialg1}) we have a $\kappa(s)-$Lie algebra isomorphism $\Lie(\bG_s) \simeq \bLie(\bG)\big(\kappa(s)\big).$ From the isomorphism $\bLie(\bG) \simeq\bV( \omega^1_{\bG/\rS})$, we see that the dimension of $\bLie(\bG)\big(\kappa(s)\big)$ at a point $s \in \rS$
is the dimension of the  $\kappa(s)$--space $\Hom_{\kappa(s)}\big( \omega^1_{\bG/\rS} \otimes_{\cO_S} \kappa(s), \kappa(s)\big),$ which is the dimension of  $\omega^1_{\bG/\rS} \otimes_{\cO_S} \kappa(s)$ since $\omega^1_{\bG/\rS}$ is  an $\cO_\rS$--module of finite type. Furthermore, since $\omega^1_{G/S}$ is in fact of finite presentation, 
Lemma \ref{lem_upper} implies that the map $s \mapsto \dim_{\kappa(s)}\bigl( \bLie(\bG)(\kappa(s)) \bigr)$
is upper semi-continuous.
\end{proof}

\bigskip

\section{Group schemes over Jacobson schemes}

\subsection{Jacobson schemes and smoothness}
We refer the reader to \cite[\S 10]{EGAIV} for general results about Jacobson schemes. Let $\rX$ be a topological space. We recall that a subset $\rX_0$ of $\rX$ is called {\it very dense} if for every closed subset $\rZ \subset \rX$ we have $\rZ = \overline{\rZ \cap \rX_0}.$ Let $\rS$ be a scheme. View $\rS$ as a topological space and let $\rS_0 \subset \rS$ the set of closed points. The scheme $\rS$ is said to be {\it Jacobson} if $\rS_0$ is a very dense subset of $\rS.$ 

Let us recall for the sake of completeness that:

(i) Very dense subsets of a topological space are dense.

(ii) If  $\rS= \Spec(R)$, then $\rS$ is Jacobson if and only if $R$ is a Jacobson ring, that is every ideal of $R$ is an intersection of maximal ideals.

(iii) If $\rS= \bigcup_{i \in I} \Spec(R_i)$, the $\rS$ is Jacobson if and only if every $R_i$ is a Jacobson ring.

 (iv) Every radical ideal $I$ is 
 the intersection of the maximal ideals containing it  \cite[Tag 00G4]{Stacks}.\footnote{AP. We will not use these properties in what follows. }

\begin{proposition} Let $f : \rX \to \rS$ be a scheme morphisms which is locally of finite presentation. Assume that $x \in \rX$ is such that:
\medskip

(i) $f$ is flat at $x$.

(ii) $f_{\kappa(x)}$ is smooth.
\medskip

\noindent Then $f$ is smooth on an open neighborhood of $x.$
\end{proposition}

\begin{proof} By assumption there exists affine open subschemes $\Spec(A) = \rU \subset \rX$ and $\Spec(B) = \rV \subset \rS$ such that $x \in \rU,$ $f(\rU) \subset \rV$ and $A$ is a finitely presented $B$-algebra. Let $\fq \in \Spec(A)$ and $\fp \in \Spec(B)$ be the points corresponding to $x$ and $f(x)$ respectively.

By (i) $A_\fq$ is flat over $B_\fp$, while by (ii) $A \otimes_B \kappa(\fp)$ is smooth over $\kappa(\fp).$ It follows from \cite[Tag 00TF]{Stacks}  that the restriction of $f$ to $\rU$ is smooth at $\fq$, hence (by definition of smoothness) smooth also in a neighbourhood of $\fq.$
\end{proof}

\begin{corollary}\label{cor_smooth_jacobson} Let $f : \rX \to \rS$ be as above. Assume that (i) and (ii) hold for all closed points of $\rX$. If $\rS$ is Jacobson, then $f$ is smooth.
\end{corollary}
\begin{proof} By \cite[Cor. 10.4.7]{EGAIV}  $\rX$ is  Jacobson. Let $\rX_0$ be the set of closed points of $\rX$ viewed as a topological space. By the Proposition there exists an open $\rU \subset \rX$ containing $\rX_0$ in which $f$ is smooth. Let $\rZ = \rX \setminus \rU.$ If $\rZ \neq \emptyset$, then $\rZ \cap \rX_0 \neq \emptyset$ because $X_0$ is very dense. This contradicts  $\rX_0 \subset \rU.$  Thus $\rZ = \emptyset$ so that $\rU = \rX$ as desired.
 \end{proof}

\begin{proposition}  \label{prop_smooth}
Let $\rS$ be an integral  Jacobson scheme, and
$\bG$ be an $\rS$-group scheme  of finite presentation. Assume that for all closed point $b \in \rS$ the fibres
$\bG_b$ are smooth and of the same dimension $d$.\footnote{\, Since $\bG$ is of finite presentation $\bG_s$ is an algebraic $\kappa(s)$-group. Thus $\bG_s$ has a finite number of irreducible components. They all have the same (finite) dimension  and this  is also the dimension of $\bG_s.$ } Then.
\smallskip

\noindent (1) For all  $s \in \rS$ the algebraic $\kappa(s)$--group $\bG_s$ 
 is smooth of dimension $d.$
 \smallskip
 
\noindent (2) 
The $\rS$-functor $\bG^\circ$ is representable by a smooth $\rS$--group scheme
of relative dimension $d$  which is open in $\bG$.

\smallskip

\noindent (3) If for
each closed point $b$ of $\rS$ the fibre
$\bG_b$ is connected,
then $\bG$ is smooth and has connected geometric fibres (i.e., $\bG_{\overline{b}}$ is connected).
\end{proposition}

 \begin{proof} (1) We denote by $\eta: \Spec(F) \to \rS$
 the generic point of $\rS$.
 Chevalley's generic flatness theorem  \cite[10.85]{GW}
 shows that there exists an open dense subset $\rU$ of $\rS$
 such that $\bG \times_\rS \rU$ is flat over $\rU$. We know that  $\rU$ is also  Jacobson  \cite[Prop. 10.3.3]{EGAIV}.  Corollary  \ref{cor_smooth_jacobson} shows that $\bG \times_\rS \rU$ is 
 smooth over $\rU$. By \cite[Tag 05F7]{Stacks}, up to shrinking 
$\rU$ if necessary, the dimension function of the fibres of $\bG \times_\rS \rU \to \rU$ is constant. Since by assumption this dimension has value $d$ on the closed points of $\rU$,
 we get that  $\bG \times_\rS \rU$ is smooth of relative dimension $d$.\footnote{\, That is, all the fibres $\bG_s, s \in \rU$ are equidimensional and this dimension is $d.$}
 In particular the algebraic $F$-group $\bG_F$ is smooth of dimension $d$.

 Let $s \in \rS$ and let $c = \dim(\bG_s). $  According to Chevalley's semicontinuity theorem \cite[VI$_B$.4.1]{SGA3}
 $$F_c = \{x \in \rS : \dim(\bG_x) \geq c\} \subset \rS$$
 is a closed set. Since $F_c \cap \overline{ \{s\}}$ contains $s$ and is closed
 \begin{equation}\label{closedpoint}
 \overline{ \{s\}} \subset F_c.
 \end{equation}
Since $\rS$ is Jacobson, it follows that there exists a closed point $b$ in $\overline{ \{s\}}.$ We have  
\begin{equation}\label{inneq1}
\dim(\bG_s) \leq \dim(\bG_b)=d.
\end{equation}

 On the other hand $s$ belongs to $\overline{\{\eta\}}.$ The same semicontinuity reasoning used above shows that 
 \begin{equation}\label{inneq2}
 \dim(\bG_F) \leq \dim(\bG_s). 
 \end{equation}
Since $\bG_F$ is of dimension $d$ it follows from (\ref{inneq1}) and  (\ref{inneq2})  that $\dim(\bG_s)=d$.

 For establishing   smoothness we use a variation of the previous argument 
 by applying semicontinuity considerations to the Lie algebras. Lemma \ref{lem_upper_Lie}   yields the inequalities
  $$
   \dim_F \Lie(\bG_F) \leq 
  \dim_{\kappa(s)}\Lie(\bG_s)\leq \dim_{\kappa(b)}\Lie(\bG_b).$$
 Since $\bG$ is smooth of dimension $d$ at $\eta$ and at $b$ we 
 obtain that $d=\dim_{\kappa(s)}\Lie(\bG_s).$ Thus $\Lie(\bG_s)$ and $\bG_s$ have the same dimension (namely $d$). By the smoothness criterion \cite[II.5.2.1]{DG}  $\bG_s$ is smooth. This completes the proof of  (1).

 (2) By (1) we see that assumption (ii) of \cite[VI$_B$ Cor. 4.4]{SGA3} holds.  It follows that 
 the $\rS$-functor $\bG^\circ$ is representable by a smooth $\rS$--group scheme
which is open in $\bG$.

\smallskip

\noindent

(3) According to \cite[Tag 055I]{Stacks}, the level set
$$
E_n= \bigl\{ s \in \rS \mid \bG_s \enskip \hbox{has 
$n$ geometrically connected components} \bigr\}
$$
is a locally constructible subset of $\rS$
for each $n \geq 1$.
Since $\rS$ in a Jacobson scheme, 
the set $\rS_0$ of closed points of the underlying topological space of $\rS$ is very 
dense in $\rS.$ Thus $\rS_0 \cap E_n$
is  dense in $E_n$ for each $n$
\cite[10.1.2, (b')]{EGAIV}.
Our assumption implies that $E_n = \emptyset$ for $n \geq 2$ so that 
$E_1=\rS$.
 That $\bG$ is smooth now follows from (2).
\end{proof}

 \subsection{Forms}

 \begin{proposition} \label{prop_form}  
 Let $\rS$ is an integral Jacobson scheme and $\bG$ an affine $\rS$--group scheme of finite presentation.
 Assume  that $\bG_b$ is reductive and of dimension $d$ 
 for all closed point $b \in \rS$.
 Then $\bG$ is a reductive $\rS$--group scheme of constant type. In particular, there exists a unique Chevalley group $\bG_0$ such that $\bG$ is an  $\rS$--form of 
 $\bG_0 \times_\Z \rS$.
 \end{proposition}

 \begin{proof} Proposition \ref{prop_smooth}.(3) shows that  $\bG$ is smooth of relative dimension $d$
 and has connected geometric fibres.
 According to \cite[Prop. 3.1.9.(1)]{Co1} $\bG$ is reductive in a neighborhood of each 
 closed point $b \in \rS$. Since  $\rS$ is a Jacobson scheme, it follows that  
  $\bG$ is reductive. Since the type function is locally 
  constant, the connectedness of 
$\rS$ implies that  $\bG$ has constant type $t_0$ (see \S \ref{CCFT}). Let  $\bG_0$ be the corresponding Chevalley
group scheme.
According to \cite[XXIII 5.6]{SGA3} $\bG$ is an $\rS$--form of $\bG_0 \times_\Z \rS$.
 \end{proof}

 \begin{corollary} \label{cor_form} Let $\bG$ and $\bG'$ be group schemes over an integral Jacobson scheme $\rS.$
Assume that $\bG$ is reductive and that $\bG'$ is affine 
 and of finite presentation.
If $\bG'_b$ is a form of $\bG_b$  
 for each closed point $b \in \rS$,
 then $\bG'$ is an $\rS$--form of 
 $\bG$. In particular, $\bG'$ is a reductive $\rS$--group scheme \qed \end{corollary}

 \subsection{Lie algebras}
 If $\cL$ is an $\cO_\rS$--Lie algebra and $s \in \rS,$ then the $\cO_{\Spec\kappa(s)}$--algebra $\cL \otimes_{\cO_\rS} \kappa(s)$ obtained by base change is simply a $\kappa(s)$--algebra that we will denote by $\cL_s.$\footnote{\, A  $\cO_{\Spec\kappa(s)}$--algebra ``is the same" as a  $\kappa(s)$--algebra.} Let $\bL =\bW(\cL)$ and denote by a harmless abuse of notation  the $\kappa(s)$--algebra $\bL\big(\kappa(s)\big)$ by $\bL_s.$ Note that $\cL_s =\bL_s.$ 
 
 \begin{proposition} \label{prop_lie_form} Assume that $\rS$ is an integral Jacobson $\Q$--scheme.
 Let $\cL$ be an $\cO_\rS$--Lie algebra which is locally free of rank $d$, and let $\bL = \bW(\cL)$ be its corresponding $\bO_\rS$--Lie algebra.
 Assume that for each closed point $b \in \rS$ the $\kappa(b)$--Lie algebra $\bL_b$
 is  semisimple.\footnote{\, Necessarily of finite dimension $d.$} Then there exists a unique finite dimensional split semisimple Lie $\Q$-algebra $\rm{L}_0$
 such that $\bL$ is a form of $\bL_0 := \rm{L}_0 \otimes_\Q \bO_\rS$.\footnote{\, By definition $\bL_0(\T)$ is the $\cO_\T(\T)$-Lie algebra ${\rm L}_0\otimes_\Q \cO_\T(\T).$}
 \end{proposition}

 \begin{proof} (1) By Lemma \ref{repres} $\bG :=\bAut_{\bO_\rS {\rm - Lie}}(\bL) $ is an affine $\rS$--group of finite presentation. For each closed point $b \in \rS$, the algebraic group
 $\bG_b= \bAut_{\kappa(b)-{\rm Lie}}(\bL_b)$ is smooth of dimension $d$
 and $\bG_b^\circ$ is semisimple adjoint.
 Proposition \ref{prop_smooth}.(2) shows that 
 $\bG^\circ$ is representable by an open subgroup scheme of $\bG$ which 
 is smooth of relative dimension $d.$ Appealing now to Proposition \ref{prop_form}
 yields that $\bG^\circ$ is adjoint semisimple and is a form of the $\rS$-group corresponding to a (unique) semisimple adjoint Chevalley group scheme
 $\bG_0^\circ.$  We denote 
 by $\rm{L}_0$ the $\Q$--Lie algebra of $\Lie(\bG_0^\circ) \otimes_\Z \Q.$ This yields the $\bO_\rS$-Lie algebra  that we denoted by $\bL_0.$ Observe
 that $\bLie(\bG)$ is a form of $\bL_0$ as an $\bO_\rS$-Lie algebras. 
 
  By Lemma \ref{MZ} we have
 $\bL \simeq \bDer_{\bO_\rS-\rLie}(\bL).$ 
 Finally 
since $\bG \simeq \bAut_{\bO_\rS-\Lie}(\bL)$ we can apply Lemma \ref{derLie} to conclude that
$ \bLie(\bG) \simeq \bDer_{\bO_\rS-\rLie}(\bL).$ It follows that $\bL$ is a form of $\bL_0$ as desired.
\end{proof}

The following variant of the ideas presented heretofore will allow us to answer Burban's question (see \S \ref{Burban} below). 
 
\begin{proposition} \label{prop_lie_form2}
Assume that $\rS$ is an integral Jacobson $\Q$--scheme.
Let $\bG$ be a semisimple adjoint $\rS$--group scheme and consider its $\cO_\rS$--Lie algebra 
$\cL=\cLie(\bG)$.
 
Let $\cL'$ be a sheaf of $\cO_\rS-$Lie algebras. We assume that as an $\cO_\rS$--module
$\cL'$ is locally free of rank $d,$ and  that for each 
closed point $b\in \rS$,  $\cL'_b$ is a form of  
$\cL_b$.

\smallskip

\noindent (1) The $\rS$--functor $ \bAut_{\cO_\rS - {\rm Lie}}(\cL')$ is a smooth affine 
 $\rS$--group scheme 
whose connected  component of the identity $\bG'$  is an $\rS$--form of $\bG$. In particular $\bG'$ is a semisimple adjoint $\rS$--group scheme. 

\smallskip

\noindent (2) There is a natural isomorphism $\cLie(\bG') \simeq \cL'$.

\smallskip

\noindent (3) $\cL'$ is an $\rS$--form of $\cL$.

\smallskip

\noindent (3 bis)  $\bL'$ is an $\rS$--form of $\bL,$ where $\bL' = \bW(\cL')$ and $\bL = \bW(\cL)$.

\end{proposition}

\begin{proof} (1)   Let $\bL' = \bW(\cL').$ By Lemma \ref{repres}  $\bH := \bAut_{\cO_\rS - {\rm Lie}}(\cL') \simeq \bAut_{\bO_\rS - {\rm Lie}}(\bL')$ is an affine $\rS-$group scheme of finite presentation.  Since $\bH_\T = {\bAut_{\bO_\T - {\rm Lie}}(\bL'_\T)}$ for all  $\rS$--scheme $\T,$ we have $\bH_s = \bAut(\bL'_s)$ for all $s \in \rS,$ where $ \bL'_s :=  \bL' \otimes_{\cO_\rS} \kappa(s).$ These algebraic groups are smooth since $\kappa(s)$ is of characteristic $0.$
 
Let $b \in \rS$ be a closed point. Since the $\kappa(b)$--Lie algebra $\cL'_b$ is a twisted form of $\cL_b$, the corresponding $\kappa(b)$-algebraic group $\bAut(\cL'_b)$ is a twisted form of $\bAut(\cL_b).$ Similarly for their connected component of the identity. As we have seen that $\bAut(\cL'_b) = \bH_b.$ On the other hand since $\bG$ is semisimple adjoint $\bAut(\cL_b)^\circ = \bG_b.$ This yields that $\bH_b^\circ$ is a twisted form of $\bG_b$. In particular all the $\bH_b$ are smooth groups of the same dimension. Proposition \ref{prop_smooth}.(2) then shows that the $\rS$--subfunctor $\bH^\circ$ 
 of $\bH$ is representable by a smooth affine $\rS$--group scheme $\bG'$. 
 
 According to \cite[prop. 3.1.9.(1)]{Co1},  $\bG'$ is reductive on a neighborhood of each 
 closed point $b \in \rS$. Since  $\rS$ is a Jacobson scheme, it follows that  
  $\bG'$ is reductive. Since $\rS$ is connected the type of $\bG'$ is constant. Since $\bG'$ and $\bG$ have the same type at closed points, their (constant) types coincide. Thus $\bG'$ is a twisted form of $\bG$ \cite[XXIII 5.6]{SGA3}. In particular, $\bG'$ is semisimple adjoint.

 \smallskip
 
 \noindent (2) 
  By Lemma \ref{derLie} 
$$\bDer_{\bO_\rS-{\rm alg}}(\bL') \simeq \bLie\big(\bAut_{\bO_\rS-{\rm alg}}(\bL')\big) = \bLie(\bH^\circ) = \bLie(\bG').$$

\smallskip 

On the other hand we have an isomorphism of $\bO_\rS$--Lie algebras
$\bL' \simeq \bDer_{\bO_\rS-{\rm alg}}(\bL')$ (Lemma \ref{MZ}), whence an isomorphism 
$\bL' \simeq \bLie(\bG')$.
It follows that $\bW(\cL') \simeq \bW(\cLie(\bG'))$.
Since $\bW$ is full and faithful $\cL'$ and  $\cLie(\bG')$
 are isomorphic $\cO_\rS$--Lie algebras.

 \smallskip

 \noindent (3) From (2) and Remark \ref{Liesmooth} we get that $\bLie(\bG') = \bW\big(\cLie(\bG')\big)$ is a form of $\bLie(\bG) = \bW\big(\cLie(\bG)\big).$ It follows that $\cLie(\bG')$ is a form of $\cLie(\bG).$ Thus (3) follows from (2).
 
 \smallskip
 
\noindent (3 bis.) Follows from (3) since $\bW$ commutes with base change.
 \end{proof}

 \begin{section} {Burban's question}\label{Burban}

 In order to formulate Burban's question we need to recall the concept of loop algebra of a simple Lie finite dimensional complex Lie algebra $\gg.$
 
 \noindent Let $R= \C[t^{\pm 1}].$ Fix a positive
integer $d,$ and set $R_d = \C[t^{\pm
\frac{1}{d}}].$ The natural map $R\to R_d$ is  faithfully flat and finite \'etale. Let 
$\xi \in \C$ be a primitive $d$-th root of unity. Then the elements of $\Gamma = \Z/d\Z$ act as automorphisms of $R_d$ over $R$ via $$
^{\ol{e}} t^{\frac{1}{d}}= \xi  ^e
t^{\frac{1}{d}}.
$$
for $e \in \Z.$ This action makes $R_d$ into a Galois extension of  $R$ with Galois group $\Gamma.$

Let  $\sigma$ be an automorphism of $\gg$ of order $d.$   For $i \in \Z$ consider the
 eigenspace
$$
\gg_i =\{x\in \gg:\sigma(x) = \xi  ^ix\} $$

Then  $\gg = \us{0 \leq i <d}\oplus \gg_i$.
Out of this data we define the corresponding 
 {\it loop algebra} 
\begin{equation}\label{defmulti}
 L(\gg,\sigma) = \us{i \in \Z}\oplus\,
\gg_i \otimes t^{\frac{i}{d}} \subset \gg\otimes _\C R_d.
\end{equation}


The simple but crucial observation is  that $ L(\gg,\sigma)$  is stable under the scalar action of  $R.$ Thus $L(\gg,\sigma)$  is not only an infinite dimensional complex Lie algebra, but also an $R$-Lie algebra. As we shall see, it is the algebra structure over this ring that allows non-abelian cohomological considerations to enter into the picture. 

It is an easy linear algebra exercise to verify that we have a natural $R_d$--algebra isomorphism

\begin{equation}\label{iso1} L(\gg, \sigma) \otimes_R R_d \simeq \gg \otimes_\C R_d \simeq \gg_R \otimes_R  R_d.
\end{equation}
where  $\gg_R := \gg \otimes_\C R.$
This shows that the $R$-Lie algebra $ \rL := L(\gg, \sigma)$ is a twisted form of $\gg \otimes_\C R.$ It therefore corresponds to an $\bAut(\gg_R)$--torsor $\bX$ over $\Spec(R).$  More precisely $\bX = \bIsom_{R-{\rm Lie}}(\gg_R, \rL).$\footnote{\, See \cite[III \S4]{DG} for the material on torsors used in this section.} 

We can now formulate Burban's questions (essentially verbatim except for some notation changes):

{\bf Question 1:} Let $\rL'$ be a Lie algebra over the ring $\C[t]$. Assume $\rL'$ is free as a module and  that for
any complex number $b$ the quotient Lie algebra $\rL'/(t-b)\rL'$ is isomorphic to
$\gg.$ Does it follow that $\rL'$ is isomorphic to $\gg \otimes_\C \C[t]$ (as a Lie algebra over $\C[t]$)?

{\bf Question 2:} Similarly, let  $\rL'$ be a Lie algebra over the ring $\C[t^{\pm 1}]$ which we assume is free as a module and such that for any $b \in \C^\times$ the quotient $\rL'/(t-b)\rL'$ is
isomorphic to $\gg.$ Does it follow that $\rL'$ is isomorphic to a  loop
algebra (with respect to an automorphism of $\gg$ of finite order)?

\medskip

To answer these questions we take $\rS =\Spec(R)$ where $R =\C[t]$ or $\C[t^{\pm 1}]$, and where $\cL'$  and $\cL$ are 
the $\cO_\rS$--Lie algebra corresponding to $\rL'$ and $\gg \otimes_\C R$ respectively. By Proposition \ref{prop_lie_form2} we see that $\cL'$ corresponds to a torsor $\bX'$ over $\Spec(R)$ whose class is an element of $H^1(R, \bAut(\gg_R)).$

Recall (\cite[XXIV 1.3 and 7.3.1]{SGA3}. See also \S \ref{2.4}) the split exact sequence of $R$--group schemes
\begin{equation}\label{splitexact}
1 \to \bG \to \bAut(\rL) \to \bOut(\rL) \to 1
\end{equation}
where $\bG$ is the split adjoint semisimple  $R$-group scheme corresponding to $\gg$, and $\bOut(\rL)$ is the constant $R$-group scheme corresponding to the finite (abstract) group ${\rm Out}(\gg)$ of the symmetries of the Coxeter-Dynkin diagram of $\gg.$ This allow us to compute the relevant $H^1$ an thus determine the nature of $\bX'$, hence $\rL'.$
According to \cite[Cor. 3.3]{CGP}, we have a bijection
$$
H^1(R,\bAut(\rL)) \simlgr  H^1(R,\bOut(\rL)).
$$

\medskip

{\bf Answer 1:} If  $R =\C[t]$ then  $H^1\big(R, \bOut(L))=1$ since $R$ is simply connected. It follows that $H^1\big(R, \bAut(\rL)\big) = 1$ and therefore that $L' \simeq \gg \otimes_k R$. 
\medskip

{\bf Answer 2:} If $R = \C[t^{\pm 1}]$ then furthermore 
$H^1(R, \bOut(\rL))$ is  the set of conjugacy classes of the (abstract) group ${\rm Out}(\gg)$ which, in terms of forms of $\rL$, correspond  to the loop algebras $L(\gg, \sigma)$ with $\sigma \in {\rm Out}(\gg)$  (see \cite{P2} for details, or more generally \cite{P3}). Thus $\rL'$  is a loop algebra. 

 \end{section}

\end{document}